\documentclass[preprint,12pt]{elsarticle}
\usepackage[utf8]{inputenc}
\usepackage[T1]{fontenc}

\usepackage{ifthen}

\ifdefined \bDoNotIncludePackages \else
  \newcommand{\bDoNotIncludePackages}{0}
\fi
\ifdefined \bSkipDocumentSetting \else
  \newcommand{\bSkipDocumentSetting}{0}
\fi
\ifdefined \bDoNotDefineTheorems \else
  \newcommand{\bDoNotDefineTheorems}{0}
\fi

\ifthenelse{\equal{\bDoNotIncludePackages}{1}}{}
{
   \usepackage{graphicx}
   \usepackage{amsmath,amsthm,amssymb,mathtools}
   \usepackage{enumerate}
   \usepackage{cleveref} 
   \usepackage{subfig}
   \usepackage{color}
}

\ifthenelse{\equal{\bSkipDocumentSetting}{1}}{}{

\addtolength{\voffset}{-1cm} 
\addtolength{\hoffset}{-1.5cm} 
\setlength{\textheight}{22cm} \setlength{\textwidth}{16cm}
}

\def\N{\mathbb N}
\def\A{\mathcal A}

\def\C{\mathcal C}

\def\P{\mathcal P}

\def\P{\mathcal P}
\def\PT{{\mathcal P}_{\Theta}}
\def\DT{{D}_{\Theta}}

\def\L{\mathcal L}
\def\Lu{{\mathcal L}(\uu)}
\def\uu{\mathbf u}

\def\tt{\mathbf t}

\def\PalT{{\rm Pal}_{\Theta}}

\def \Rk#1 {$\mathcal{R}_{#1}$}
\def \Rext {{\rm Rext}}
\def \Lext {{\rm Lext}}

\def \b {{\rm b}}

\def \FC#1 {
\mathcal{C}
\ifthenelse{\equal{#1}{}}{}{(#1)}
}

\def \PC#1 {
\mathcal{P}_{\Theta}
\ifthenelse{\equal{#1}{}}{}{(#1)}
}

\def \PCn#1 {
\mathcal{P}
\ifthenelse{\equal{#1}{}}{}{(#1)}
}

\def \TuT {T_{\Theta}}

\def \Tr {R}
\def \Ta {\Theta_1}
\def \Tb {\Theta_2}
\def \gT {\gamma_{\Theta}}

\ifthenelse{\equal{\bDoNotDefineTheorems}{1}}{}
{
\newtheorem{thm}{Theorem}

\newtheorem{corollary}[thm]{Corollary}

\newtheorem{proposition}[thm]{Proposition}

\newtheorem{defi}[thm]{Definition}

\crefname{thm}{theorem}{theorems}
\crefname{thrm}{theorem}{theorems}
\crefname{coro}{corollary}{corollaries}
\crefname{example}{example}{examples}
\crefname{lem}{lemma}{lemmas}
\crefname{lmm}{lemma}{lemmas}
\crefname{claim}{claim}{claims}
\crefname{obs}{observation}{observations}
\crefname{proposition}{proposition}{propositions}
\crefname{prop}{proposition}{propositions}
\crefname{defi}{definition}{definitions}

\crefname{theorem}{theorem}{theorems}
\crefname{corollary}{corollary}{corollaries}
\crefname{example}{example}{examples}
\crefname{lemma}{lemma}{lemmas}
\crefname{proposition}{proposition}{propositions}
\crefname{definition}{definition}{definitions}

\theoremstyle{remark}
\newtheorem{remark}[thm]{Remark}
\newtheorem{example}[thm]{Example}

\crefname{example}{example}{examples}
}

\definecolor{sediva}{rgb}{0.85,0.85,0.85}

\begin{document}

\begin{frontmatter}



\title{Languages invariant under more symmetries: \\
overlapping factors versus palindromic richness}


\author[fnspe]{Edita Pelantov\'a}
\author[fnspe,fit]{\v St\v ep\'an Starosta}
\address[fnspe]{Department of Mathematics, FNSPE, Czech Technical University in Prague,\\ Trojanova 13, 120~00 Praha~2, Czech Republic}
\address[fit]{Department of Applied Mathematics, FIT, Czech Technical University in Prague, \\ Th\'akurova~9, 160~00, Praha~6, Czech Republic}

\begin{abstract}

Factor complexity $\mathcal{C}$  and palindromic complexity $\mathcal{P}$ of
infinite  words with language closed under reversal are known to
be related  by the inequality $\mathcal{P}(n) + \mathcal{P}(n+1)
\leq 2 + \mathcal{C}(n+1)-\mathcal{C}(n)$ for any $n\in
\mathbb{N}$\,. Words for which the equality is attained for any
$n$ are usually called rich in palindromes. We show that rich words contain infinitely many overlapping factors.
We study words whose languages are invariant under a finite group $G$
of symmetries. For such words we prove a stronger version of the
above inequality. We introduce the notion of $G$-palindromic richness
and  give several examples of $G$-rich words, including the
Thue-Morse word as well.

\end{abstract}

\begin{keyword}


palindromic richness \sep overlaps \sep symmetries \sep Thue-Morse
word \sep group palindromic richness

\MSC 68R15

\end{keyword}

\end{frontmatter}

\section{Introduction}

In the last decade, a broad interest in the study of palindromes can
be observed. Attention to palindromes was brought on one hand by
the article \cite{DrJuPi} where a bound on the number of distinct palindromes
occurring in a finite word was given, and on the other hand
by the role played by palindromes in the spectral theory of Schr{\"o}dinger
operators with aperiodic potential \cite{HoKnSi}.
The fact that the existence of many palindromes in an infinite word
$\uu$ is connected with its factor complexity $\C$ was for the first time recognized in the article
\cite{AlBaCaDa}. Its authors proved that
\begin{equation}\label{ABCD}
\mathcal{P}(n) \leq \tfrac{16}{n} \Big ( \mathcal{C}(n)+
\mathcal{C}\bigl(\bigl\lfloor\tfrac{n}{4}\bigr\rfloor\bigr) \Big)
\end{equation}
where $\P$ counts the number of
distinct palindromes of given length occurring in $\uu$.  A special case of
this inequality for fixed points of primitive morphisms was
already proven in \cite{DaZa}.  In \cite{BaMaPe}, a
relation between palindromic complexity and increment of factor
complexity was established for infinite words whose language is
closed under reversal:
\begin{equation}\label{nasa}
\mathcal{P}(n) + \mathcal{P}(n+1)   \leq 2 +
\mathcal{C}(n+1)-\mathcal{C}(n)\,.
\end{equation}
In \cite{BaMaPe}, the relation is stated under an additional hypothesis of uniform recurrence. However, this hypothesis is not used in the proof and the claim is valid without it. 
The words for which the equality in \eqref{nasa} is attained for
any $n \in \N$ are called rich or full and there exists
extensive literature on the topic, see \cite{DrJuPi,GlJuWiZa}.
The most famous examples of rich words are episturmian words (see \cite{DrJuPi}), which include Sturmian and Arnoux-Rauzy words,
and words coding interval exchange transformations determined by a symmetric permutation, see \cite{BaMaPe}.

Palindromes are formally defined as fixed points of the reversal
mapping. We may  replace the reversal mapping by another
involutive antimorphism $\Theta$ to obtain a generalization of the
notion of a palindrome. A fixed point of such mapping $\Theta$ is called a
$\Theta$-pa\-lindrome.  In \cite{Sta2010}, the second author showed that the
inequality \eqref{nasa} is still valid  if one substitutes
palindromic complexity $\mathcal{P}(n)$ by $\Theta$-palindromic
complexity $\mathcal{P}_{\Theta}(n)$. Moreover, many properties
known to be  possessed  by  rich words are possessed  by  their
generalizations -  $\Theta$-rich words - as well.

In this article we introduce the new concept of $G$-rich words,
where $G$ is a finite group generated by more antimorphisms.
Motivation for this new notion  is a property of  \mbox{$\Theta$-rich}
words we show in Theorem \ref{ctverce}, i.e., that any
$\Theta$-rich word contains infinitely many overlapping factors.
Therefore, words without large overlaps have no chance to be  rich
in $\Theta$-palindromes. But paradoxically, the language of the most prominent
word without overlaps - the Thue-Morse word - is closed
simultaneously  under two antimorphisms and  contains infinitely
many palindromes and $\Theta$-palindromes.

For an infinite word whose language is closed under two commuting
antimorphisms  $\Theta_1$ and $\Theta_2$ we deduce   a new
inequality relating  factor complexity $\mathcal{C}(n)$ with
$\mathcal{P}_{\Theta_1}(n)$ and $\mathcal{P}_{\Theta_2}(n)$, see
Theorem \ref{nerovnostProDva}. We also show that for the
Thue-Morse word   the equality in Theorem \ref{nerovnostProDva}
holds  for each $n$, i.e.,  that the Thue-Morse word is saturated
 by palindromes and $\Theta$-palindromes together  up to the highest
possible level. Therefore, in Section \ref{grupasymetrii}, we
propose to adopt a new definition of richness for words whose
language is closed under all elements of a finite group $G$ of symmetries (formed
by morphisms and antimorphisms). To such a word we assign a graph
of symmetries $\Gamma_n$ for any $n\in\mathbb{N}$.   The
connectedness of $\Gamma_n$
 implies an inequality between  factor complexity and
$\Theta$-palindromic complexities, see Theorem \ref{nerovnostProVice}.
Nevertheless,  the  definition of $G$-richness  is not based on an
inequality, but on the structure of the graph of symmetries. Let
us stress that in the case when the group $G$ is generated only by one
antimorphism the new definition of richness and the old one
coincide. In Section \ref{sec:examples}, we provide some examples of
$G$-rich words.

The new point of view on richness of infinite words triggers many
questions on properties of rich words. Some of these open questions are
collected in the last section.

\section{Preliminaries}
An \textit{alphabet} $\mathcal{A}$ is a finite set, its elements are
usually called \textit{letters}. By a \textit{finite word} over an alphabet
$\mathcal{A}$ we understand a finite string $w = w_1w_2\ldots w_n$
of letters $w_i\in \mathcal{A}$. Its \textit{length} $n$ is denoted by
$|w|$. The set of all finite words over $\mathcal{A}$ equipped
with the operation of concatenation is the \textit{free monoid} $\mathcal{A}^*$, its
neutral element is the \textit{empty word} $\varepsilon$.  A word $v \in
\mathcal{A}^*$ is a \textit{factor} of a word $w \in \mathcal{A}^*$ if there
exist words $s,t\in \mathcal{A}^*$ such that $w=svt$. If
$s=\varepsilon$, then $v$ is a \textit{prefix} of $w$; if  $t=\varepsilon$,
then $v$ is a \textit{suffix} of $w$.

\subsection{Antimorphisms and their fixed points}
A  mapping $\varphi$ on $\mathcal{A}^*$ is
\begin{itemize}
\item
a \textit{morphism} if $ \varphi(vw)=  \varphi(v) \varphi(w)$ for any $v,w
 \in \mathcal{A}^*$;
\item an \textit{antimorphism} if $ \varphi(vw)=  \varphi(w) \varphi(v)$ for
any $v,w \in \mathcal{A}^*$.
\end{itemize}
We denote the set of all morphisms and antimorphisms on $\mathcal{A}^*$ by $AM({\mathcal{A}^*})$.
Together with composition, it forms a monoid with the identity mapping $\rm{Id}$ as the unit element.
 The set of all morphisms, denoted by $M({\mathcal{A}^*})$, is a
submonoid of $AM({\mathcal{A}^*})$. The \textit{reversal mapping}
$\Tr$ defined by $$\Tr(w_1w_2\ldots w_n) =
w_nw_{n-1}\ldots w_2w_1, \qquad \text{ where $w_i \in \A$,}$$ is an \textit{involutive} antimorphism, i.e.,
$\Tr^2 = \rm{Id}$. It is obvious that any antimorphism  is a
composition of $\Tr$ and a morphism. Thus
$$AM({\mathcal{A}^*})= M({\mathcal{A}^*})\cup
\Tr \big ( M({\mathcal{A}^*}) \big)\,.$$

A morphism or antimorphism $\nu \in AM(\A)^*$ is \textit{non-erasing} if for all $a \in \A$ we have $| \nu(a) | > 0$.

A fixed point of a given antimorphism $\Theta$  is a \textit{$\Theta$-palindrome}, i.e., a word $w$ is a  $\Theta$-palindrome if $w=\Theta(w)$.
If $\Theta$ is the reversal mapping $\Tr$, we say
palindrome or classical palindrome  instead of $\Tr$-palindrome.
If a non-erasing antimorphism $\Theta$ has a fixed point $w$ containing all letters of $\A$,
then, since $\Theta^2$ is a non-erasing morphism with a fixed point $w$ containing all letters of $\A$, we have $\Theta^2(a) = a$  for all $a \in \A$ . It means that $\Theta$ is
an involution, and thus a composition of $\Tr$ and an involutive permutation of letters.

Suppose $\Theta$ is an involutive antimorphism until stated otherwise.
The set  of  $\Theta$-pa\-lindromic factors of a word $w$ is denoted
by $\PalT(w)$. The cardinality of $\PalT(w)$ is bounded by
\begin{equation}\label{eq:horni_mez_poctu_pali}
\# \PalT(w) \leq |w| + 1 - \gT(w),
\end{equation}
where $\gT(w) := \# \big \{ \{ a, \Theta(a)\} \ \big | \ a \in \A, a \text{
occurs in } w \text{ and } a \neq \Theta(a) \big \}$, see
\cite{DrJuPi} for classical palindromes and
\cite{Sta2010} for generalized palindromes.

\subsection{Factor and palindromic complexities }

An \textit{infinite word} $\uu$ over an alphabet $\mathcal{A}$ is a
sequence $\uu=(u_n)_{n\in \mathbb{N}}\in \mathcal{A}^\mathbb{N}$.
We will always implicitly suppose that $\A$ is the smallest
possible  alphabet for $\uu$, i.e., any letter of $\A$ occurs at
least once in $\uu$.

 A finite word $w$ is a \textit{factor} of $\uu$ if there exists an index
$i\in \mathbb{N}$, called \textit{occurrence} of $w$, such that $w=
u_iu_{i+1}\ldots u_{i+|w|-1}$. The set of all factors of $\uu$
of length $n$ is denoted  $\mathcal{L}_n(\uu)$. The \textit{language of
an infinite  word}  $\uu$ is the set of all its factors
$\mathcal{L}(\uu)= \cup_{n\in\mathbb{N}}\mathcal{L}_n(\uu)$. An
infinite word $\uu$ is \textit{recurrent} if every its factor has at least
two occurrences in $\uu$. A factor $v\in\mathcal{L}(\uu)$ is a
\textit{complete  return word} of a factor $w$ if  $w$ occurs in $v$
exactly twice, as  a suffix and a prefix of $v$. A complete return
word $v$ of $w$ can be written as $v=qw$ for some factor $q$ which
is called a \textit{return word} of $w$. If every factor $w$ of $\uu$ occurs infinitely many times and
has only finitely many return words, then $\uu$ is \textit{uniformly recurrent}. Equivalently, $\uu$ is uniformly recurrent if every its factor occurs infinitely many times and
the gaps between its consecutive occurrences are bounded.

The \textit{factor complexity} of $\uu$ is the mapping $\mathcal{C}:
\mathbb{N}\mapsto \mathbb{N}$ defined by
$$\mathcal{C}(n)=\# \mathcal{L}_n(\uu)\,.$$
To evaluate the factor complexity of an infinite word one may consider possible
extensions of factors. A letter $a\in  \mathcal{A}$ is a \textit{left extension} of a factor $w$ in $\uu$  if $aw$ belongs to
$\mathcal{L}(\uu)$. The set of all left extensions of $w$ is
denoted $\Lext (w)$. A factor $w \in \mathcal{L}(\uu)$ is \textit{left special} (LS), if  $\#\Lext (w)\geq 2$. Analogously, we define
\textit{right extension}, the set $\Rext (w)$, and \textit{right special} (RS).  If $w$ is both right and left
special, it is \textit{bispecial} (BS). The first difference of the
factor complexity of a recurrent word  $\uu$ satisfies
\begin{equation}\label{DeltaC} \Delta \mathcal{C}(n) = \mathcal{C}(n+1) -
\mathcal{C}(n) = \sum_{w\in \mathcal{L}_n(\uu)}  \bigl(\# \Lext
(w) -1\bigr) =\sum_{w\in \mathcal{L}_n(\uu)}  \bigl(\# \Rext (w)
-1\bigr).
\end{equation}
The second difference of the factor complexity can be expressed using
the \textit{bilateral order} of a factor. It is the quantity $\b(w) = \# \{ awb
\ | \ awb \in \L(\uu)\} - \# \Lext (w) -\# \Rext (w)+1\,.$ In
\cite{Ca}, the formula
\begin{equation}\label{DeltaNaDruhuC} \Delta^2 \mathcal{C}(n) = \Delta \mathcal{C}(n+1) -
\Delta \mathcal{C}(n) = \sum_{w\in \mathcal{L}_n(\uu)} \b(w)
\end{equation}
is deduced.

The \textit{$\Theta$-palindromic complexity} of $\uu$ is the mapping $\PT(n): \N \mapsto \N$ defined by
$$
\PT(n) = \# \{ w \in \L_n(\uu) \mid w = \Theta(w) \}.
$$

\subsection{$\Theta$-richness}

A finite word $w$ is \textit{$\Theta$-rich} if the equality in \eqref{eq:horni_mez_poctu_pali} holds.
An infinite word is \textit{$\Theta$-rich} if all its factors are $\Theta$-rich.


As we already mentioned, in \cite{BaMaPe}, an inequality
involving factor and $\Tr$-palindromic complexities of infinite
words with languages closed under reversal was shown. It was
generalized for an arbitrary involutive antimorphism in
\cite{Sta2010}. In particular, it is shown that if an infinite
word has its language closed under $\Theta$, then the following
inequality holds
\begin{equation} \label{eq:nerovnost}
 \Delta  \C (n) + 2 \geq \PT(n) + \PT(n+1) \quad \text{   for all } n
\geq 1.
\end{equation}

The difference between the left and the right side in
fact decides about \mbox{$\Theta$-richness}.  Let us denote - in
accordance with the notation introduced in \cite{BrRe-conjecture} -
the quantity $\TuT(n)$ as
$$
\TuT(n) = \Delta \C (n) + 2 - \PT(n+1) - \PT(n).
$$
The following theorem  was shown in \cite{BuLuGlZa} for $\Tr$
and in \cite{Sta2010} for an arbitrary antimorphism.
\begin{thm}\label{pseudo_rich_rovnost}
If $\uu$ is an infinite word with language closed under $\Theta$,
then $\uu$ is $\Theta$-rich if and only if
$\TuT(n) = 0$ for all $n \geq 1$.
\end{thm}

\subsection{$\Theta$-palindromic defect}

The \textit{$\Theta$-palindromic defect of a
finite word} $w$, denoted $\DT(w)$, is defined as
$$
\DT(w) = |w|+1 - \gamma_{\Theta}(w) -\# \PalT(w)\,.
$$
It directly follows that $w$ is $\Theta$-rich if and only if $\DT(w) = 0$.
The \textit{$\Theta$-palindromic defect of an infinite word} $\uu$ is defined as
$$
\DT(\uu) = \sup \{ \DT(w) \mid w \in \Lu \}.
$$
Again, $\uu$ is $\Theta$-rich if and only if $\DT(\uu) = 0$. We
say that $\uu$ is \textit{almost $\Theta$-rich} if its defect $\DT(\uu)$ is
finite. A close relation between  $\Theta$-rich  and almost
$\Theta$-rich words is explained in \cite{PeSta_Milano}.

The following proposition is a direct consequence of the
definition, cf. \cite{PeSta_Milano,Sta2010}.   It is analogous to
the
case of finite $\Tr$-defect treated in \cite{GlJuWiZa}. 

\begin{proposition} \label{pseudodef_ups}
If $\uu$ is an infinite word,
then the $\Theta$-defect of $\uu$ is finite
if and only if
there exists an integer $H$ such that the longest $\Theta$-palindromic suffix of any prefix $w$ of $\uu$
such that $|w| \geq H$ occurs in $w$ exactly once, except for prefixes having the form $w = pa$
with $a \in \A$ such that $\gT(p) \neq \gT(w)$, i.e., the letter $a$ or $\Theta(a)$, $a \neq \Theta(a)$, does not occur in $p$.
\end{proposition}

In \cite{BaPeSta3} and \cite{GlJuWiZa}, various properties
of words with finite $\Tr$-defect were shown. The $\Tr$-defect of
periodic words was studied in \cite{BrHaNiRe}. The following
proposition stated for  any involutive antimorphism is an analogue
of one of these properties stated for the reversal mapping.  We
provide a short proof.

\begin{proposition} \label{pseudodef_crw}
If $\uu$ is an infinite word with finite $\Theta$-defect, then
there exists an integer $H$ such that all complete return words of
any $\Theta$-palindrome of length at least $H$ are
$\Theta$-pa\-lindromes.
\end{proposition}

\begin{proof}
Let $H$ be the constant from Proposition \ref{pseudodef_ups}.
Suppose there exists a $\Theta$-palindrome $p \in \Lu$ such that $|p| \geq H$
and $p$ has a non-$\Theta$-palindromic complete return word.
Let $r$ denote the non-$\Theta$-palindromic complete return word of $p$ that occurs in $\uu$ before any other non-$\Theta$-palindromic complete return word of $p$.
Let $q$ be the prefix of $\uu$ ending with the first occurrence of $r$ in $\uu$, i.e., $q = tr$ for some word $t$ and $r$ is unioccurrent in $q$.
Denote by $s$ the longest $\Theta$-palindromic suffix of $q$.
Since $p$ is a $\Theta$-palindromic suffix of $q$, it is clear
that $|s| \geq |p|$. If $|s| = |p|$, then we have a contradiction
to the unioccurrence of $s$. If $|r| > |s| > |p|$, then we can
find at least $3$ occurrences of $p$ in $r$ which is a
contradiction to $r$ being a complete return word of $p$.
The equality $|r| = |s|$ contradicts the fact that we
supposed $r$ to be non-$\Theta$-palindromic.
Finally, if $|r| < |s|$, then we can find an occurrence of $\Theta(r)$ which is a non-$\Theta$-palindromic complete return word of $p$ and we have a contradiction to the choice of $r$.
\end{proof}

\section{Overlapping factors in infinite words with finite $\Theta$-defect}

Infinite words with finite $\Theta$-defect contain a lot of $\Theta$-palindromes.
We show that they are very rich in overlapping factors as well.

\begin{thm}\label{ctverce}

If $\uu$ is a recurrent word with finite $\Theta$-defect, then
$\uu$ contains  infinitely many overlapping  factors, i.e., the
set
$$
\left \{ www' \in \L(\uu) \mid w' \text{ is a non-empty prefix of
} w \right \}
$$
is infinite.
\end{thm}

\begin{proof}
According to Proposition \ref{pseudodef_crw} there exists  an integer $H$
such that
\begin{itemize}
\item there exists a $\Theta$-palindrome $p_0\in \L(\uu)$ of length $|p|>H$,
 \item any complete return word of a $\Theta$-palindromic factor $v \in
\L(\uu)$ of length $|v| \geq H$ is a $\Theta$-palindrome.
\end{itemize}
We use the $\Theta$-palindrome $p_0$ as the starting element of a
sequence of  factors $(p_n)$ constructed in the following way:\\
\centerline{ for any $n\in \mathbb{N}$, $p_{n+1}$  is a
complete return word of $p_n$.}

\medskip

\noindent  Since the word $\uu$ is recurrent, any factor has at
least one complete return word and therefore our construction is
correct.

Let $L \geq H $ be an arbitrary integer.
We are going to find a factor $w$, $|w| \geq L$, such that $www'
\in \Lu$, with $w'$ being a non-empty prefix of $w$.

Since the set $\{ v \in \L_{2L}(\uu) \cup \L_{2L+1}(\uu) \mid v =
\Theta(v) \} $ is finite, there are indices $k$ and $\ell$, $k <
\ell$, such that $p_k$ and $p_\ell$ have the same central
$\Theta$-palindromic factor of length $2L$ or $2L+1$. Let $p$
 be that central factor. Let us recall that $p\in
\L(\uu)$ is a \textit{central factor} of a $\Theta$-palindrome
$v$ if $v=wp\Theta(w)$ for some finite word $w$.

According to the construction of the sequence $(p_n)$, $p$ occurs
in $p_\ell$ at least $3$ times. Let $r$ denote the return word of
$p$ occurring at the rightmost occurrence of $p$ before the
central occurrence of $p$ in $p_\ell$. Since $p_\ell$ and $rp$ are
$\Theta$-palindromes, it is clear that $rp\Theta(r) = rrp$ is a
central factor of $p_\ell$.

If $|r| \geq |p|$, then $p$ is a prefix of $r$.
Since $rrp \in \L(\uu)$, we set $w = r$ and we have directly $|w| \geq 2L$.

If $|r| < |p|$, then there exist an integer $j \geq 3$ and a factor $y$ such that $0 < |y| \leq |r|$,
$rrp = r^jy$, and $y$ is a prefix of $r$. Set $w = r^{\lfloor \frac{j}{2} \rfloor}$.
It is clear that $wwy \in \L(\uu)$.
The length of $w$ satisfies $|w| > \frac{1}{2} (j-1) |r| \geq \frac{1}{2} |p| \geq L$.

\end{proof}

\begin{remark}
Existence of squares in almost rich words was an important
ingredient in proving the Brlek-Reutenauer conjecture for
uniformly recurrent words in \cite{BaPeSta4}, where the authors
together with L. Balkov\'a  deduced a weaker form of Theorem \ref{ctverce} for uniformly recurrent words only. Let us mention
that in \cite{BrRe-conjecture}  Brlek and Reutenauer stated  the conjecture
for any word with language closed under reversal. It is
yet to be proved.
\end{remark}

\begin{remark} \label{rich_long}
Theorem \ref{ctverce} implies that any infinite $\Theta$-rich word contains an infinity of squares.
One can look for the longest finite words that are $\Theta$-rich and
do not contain a square.
For instance take $\Theta = \Tr$.
On a binary alphabet those longest words are clearly $010$ and $101$.
On a ternary alphabet there are exactly two words, up to a permutation of letters, that satisfy these conditions.
Namely $0102010$ and $0121012$.
Let $r(n)$ denote the length of such a word on an alphabet of $n$ letters.
The sequence $\big( r(n) \big )_{n=1}^{+\infty}$ begins with
$$
1, 3, 7, 15, 33, 67 \ldots
$$
To find an explicit formula for $r(n)$ remains an open question.

\end{remark}

It is  widely accepted that combinatorics on words has started its
own life with the discovery (or rediscovery) of an overlap-free word
by Axel Thue in 1912. This word, today called Thue-Morse (or
Prouhet-Thue-Morse), is the fixed point
$\lim\limits_{n\to\infty}\varphi^{n}(0)$ of the morphism
$$ \varphi( 0)= 01\quad \hbox{and} \quad  \varphi(1)=10\,.$$
The Thue-Morse word
$$
\uu_{TM} =
0110100110010110100101100110100110010110011010010110100110\ldots
$$
has language closed under reversal  and contains infinitely many
(classical) palindromes.  Moreover,  $\L(\uu_{TM})$ is closed  under
permutation of letters $0$ and $1$ and thus under a second
antimorphism $\Theta$ defined by $\Theta(0)=1$ and $\Theta(1)=0$.
The Thue-Morse word contains   infinitely many
$\Theta$-palindromes as well. Nevertheless, absence of overlapping
factors  in $ \uu_{TM}$ implies the following corollary,
which is a rephrasing of a result in \cite{BlBrGaLa}.

\begin{corollary}
The Thue-Morse word is not almost $\Theta$-rich for any
antimorphism $\Theta$.
\end{corollary}

\begin{example}\label{dvaKratRich} Let us consider the periodic word
 $\uu = (01)^\omega$ and  denote by
 $\Theta$ the antimorphism on $\{0,1\}^*$  defined by $\Theta(0)=1$ and $\Theta(1)=0$.  Obviously
\begin{itemize}
 \item  $\mathcal{C}(n) = 2$  for any $n\geq
  1$;

\item  $\mathcal{P}_{\Tr}(2n)=0$ and
$\mathcal{P}_{\Tr}(2n-1)=2$ for any for any $n\geq 1$;

\item  $\mathcal{P}_{\Theta}(2n)=2$ and
$\mathcal{P}_{\Theta}(2n-1)=0$ for any for any $n\geq 1$.
\end{itemize}
 Therefore, the periodic word
 $\uu = (01)^\omega$
is $\Tr$-rich and   $\Theta$-rich simultaneously.

\end{example}

\section{Words with language closed under two antimorphisms}

The inequality \eqref{eq:nerovnost} can be interpreted as a lower
bound on the increment  $\Delta \mathcal{C}(n)$ of the factor
complexity. The more palindromes of length $n$ and $n+1$ are contained in $\uu$, the higher the value $\Delta \mathcal{C}$.  This bound is weak when the language of a word
$\uu$ is closed under two antimorphisms.  We show that in this case
$\Delta \mathcal{C}(n)$ can be estimated more effectively.

\begin{thm}\label{nerovnostProDva}
Let $\Theta_1$ and $\Theta_2$ be two distinct commuting
involutive antimorphisms on $\mathcal{A}^*$. If $\uu$ is an
infinite word with language closed under $\Theta_1$ and
$\Theta_2$, then we have for all $ n\in \N^+$,
$$
\Delta \C (n) + 4 \geq \P_{\Theta_1}(n) + \P_{\Theta_2}(n) -
\P_{\Theta_1,\Theta_2 }(n) + \P_{\Theta_1}(n+1) +
\P_{\Theta_2}(n+1) -\P_{\Theta_1,\Theta_2 }(n+1), $$
 where $\P_{\Theta_1,\Theta_2
}(k)=\#\{ w\in \L_k(\uu)\mid w=\Theta_1(w) = \Theta_2(w)\}$.

\end{thm}
\begin{proof}
We write $w\sim v$ if $w$ is equal to $ v$ or $\Theta_1(v)$ or
$\Theta_2(v)$ or $\Theta_1\Theta_2(v)$. Since the antimorphisms
$\Theta_1$ and $\Theta_2$ are commuting, it is easy to see that
$\sim$ is an equivalence relation on $\L(\uu)$. An equivalence class
containing a factor $w$ is denoted by $[w]$.

Note that
\begin{itemize}
  \item $\#[w] = 1$,  if $w$ is  simultaneously a $\Theta_1$-palindrome and $\Theta_2$-palindrome;
  \item $\#[w] = 2$,   if $w$ is a $\Theta_1$-palindrome or
 a  $\Theta_2$-palindrome but not both;
  \item $\#[w] = 4$, otherwise.
\end{itemize}
A factor $w$ is RS or LS if and only if any factor from $[w]$ is
RS or LS.

Fix $n$. We are going to construct a directed graph
$\overrightarrow{\Gamma}= (V,\overrightarrow{E})$ with multiple edges
and loops allowed. The set $V$ of vertices of $\overrightarrow{\Gamma}$
is the set
$$V=\left \{ [w] \mid w \in \L_n(\uu), w \text{ is
special} \right \}\,.$$
 There is an edge  labeled $e \in \L(\uu)$
going from $[v]$ to $[w]$ if the prefix of $e$ of length $n$
belongs to $[w]$,  the suffix  of $e$ of length $n$ belongs to
$[v]$, and $e$ contains exactly two special factors of length $n$.
For the number of edges in $\overrightarrow{E}$ we have
$$ \# \overrightarrow{E} = \sum_{w \in \L_n(\uu), w \text{ special}} \# \Lext (w)\,.$$
Obviously
$\Theta_1(\overrightarrow{E})=\Theta_2(\overrightarrow{E})=
\overrightarrow{E}$.

Note that if $e$ is an edge between $[v]$ and $[w]$, and $\#[v] >
1$ or $\# [w] > 1$, then $\Theta_1(e)$, $\Theta_2(e)$ and
$\Theta_1\Theta_2(e)$ are also edges between $[v]$ and $[w]$ and
are distinct. In the case $\#[v] = 1$ and  $\# [w] = 1$, there are
at least two edges $e$ and $\Theta_1(e)$ between $[v]$ and $[w]$.

Let $\alpha$ denote the number of vertices $[w]$ such that $\# [w]
= 2$, $\beta$ the number of vertices $[w]$ such that $\# [w] = 4$
and $\zeta$ the number of vertices $[w]$ such that $\# [w] = 1$.

All edges in $\overrightarrow{\Gamma}$ can be divided into two
disjoint parts. We put all loops, i.e., edges starting and ending
in the same vertex, into one part. Let us denote their number by
$A$. We put all edges connecting distinct vertices into the
second part. Their number is denoted by $B$. Obviously, $\#
\overrightarrow{E} =A+B$.

\begin{description}
\item[Estimate of $B$]\quad  We exploit the  connectivity of
the graph $\overrightarrow{\Gamma}$ to give a lower bound on $B$. We
have to take into consideration that distinct connected  vertices
are connected either by four or two edges, as discussed above.
\begin{itemize}

\item  If $(\alpha+\beta)>0$ and $\zeta>0$, then $B \geq
4(\alpha+\beta-1) + 4 + 2(\zeta-1)$.

\item If $(\alpha+\beta) = 0$, then $B\geq 2(\zeta-1)$.

\item If $\zeta = 0$, then $B\geq  4(\alpha+\beta-1)$.

\end{itemize}
Altogether,

\begin{equation}\label{odhadB}B \geq 4(\alpha+\beta) + 2 \zeta -
4\,.
\end{equation}

\item[Estimate of $A$]\quad  If an edge $e$ contains a
$\Theta_i$-palindrome $p$ which is neither a prefix nor a  suffix
of $e$ of length $n$, then $\Theta_i(e)=e$. Such an edge is a
loop in the graph $\overrightarrow{\Gamma}$, and the
$\Theta_i$-palindrome $p$ is centered in $e$.

On the other hand, any $\Theta_i$-palindrome of length $n+1$ lies
on a unique edge $e$. Similarly, any $\Theta_i$-palindrome of
length $n$ which is not a special factor lies on a unique edge
$e$.  We may conclude that
$$A \geq \# \{w \in \L_n(\uu ) \mid w\  \hbox{is not special}, w= \Theta_1(w)\ \hbox{ or} \   w=
\Theta_2(w)\}$$
$$
  + \# \{w \in \L_{n+1}(\uu) \mid  w= \Theta_1(w) \
\hbox{or}\  w= \Theta_2(w)\}$$ and thus
\begin{equation}\label{odhadA} A + 2\alpha + \zeta \geq
\P_{\Theta_1}(n) + \P_{\Theta_2}(n) - \P_{\Theta_1,\Theta_2 }(n) +
\P_{\Theta_1}(n+1) + \P_{\Theta_2}(n+1) -\P_{\Theta_1,\Theta_2
}(n+1)\,. \end{equation}
\end{description}
To complete  the proof, we have to realize that
$$
\# \overrightarrow{E}  = \sum_{\substack{w \in \L_n(\uu) \\ w \text{ special}}} \# \Lext (w)
=
\sum_{\substack{w \in \L_n(\uu) \\ w \text{ special} \\ \# [w] = 4 }} \# \Lext (w)
+ \sum_{\substack{w \in \L_n(\uu) \\ w \text{ special} \\ \# [w] = 2  }} \# \Lext
(w) + \sum_{\substack{w \in \L_n(\uu) \\ w \text{ special} \\ \# [w] = 1  }} \#
\Lext (w)
$$
and thus
$$
\Delta \C(n) =  \sum_{\substack{w \in \L_n(\uu) \\ w \text{ special}}} \big ( \# \Lext (w) -
1 \big ) = \# \overrightarrow{E} - 4\beta -2\alpha - \zeta = A+B
- 4\beta -2\alpha - \zeta \,.
$$
Using the estimates \eqref{odhadA} and \eqref{odhadB},  we get the
inequality announced  by the Theorem.
\end{proof}

\begin{corollary}\label{odjistehoNplati}
If $\uu$ is a uniformly recurrent infinite word
with language closed under two distinct  commuting involutive
antimorphisms $\Theta_1$ and  $\Theta_2$, then there exists an
integer $N$ such that
$$
\Delta \C (n) + 4 \geq \P_{\Theta_1}(n) + \P_{\Theta_2}(n)  +
\P_{\Theta_1}(n+1) + \P_{\Theta_2}(n+1)\ \ \text{ for all } n >
N \,.
$$
\end{corollary}

\begin{proof}
Since $\Theta_1$ and  $\Theta_2$ are two distinct commuting
antimorphisms, their composition $\Theta_1\Theta_2$  is a
non-identical morphism which  permutes letters.  Let $a \in \A$
be a letter such that $\Theta_1\Theta_2(a) \neq a$. Since $\uu$ is
uniformly recurrent we can find an integer $N$ such that $a$
occurs in any factor longer than $N$. The equation $\Theta_1(w) =
\Theta_2(w)$ implies $\Theta_1\Theta_2(w)=w$,  which cannot be
satisfied for words longer than $N$. Thus for all $n > N$
$$
\{ w\in \L(\uu)\mid w=\Theta_1(w) = \Theta_2(w)\ \hbox{and}\
|w|=n\} = \emptyset \quad  \Longrightarrow \quad
\P_{\Theta_1,\Theta_2 }(n) = 0\,.
$$
\end{proof}

\begin{corollary}\label{skorovzdy}
If $\Theta_1$ and  $\Theta_2$ are two distinct commuting
involutive antimorphisms and $\uu$ is a uniformly recurrent
infinite word such that $\uu$ is simultaneously $\Theta_1$-rich
and $\Theta_2$-rich, then $\uu$ is periodic.
\end{corollary}
\begin{proof} The $\Theta_1$-richness and recurrence imply that $\L(\uu)$ is closed under $\Theta_1$
and for all $n\geq 1$ we have  $ \Delta \C (n) + 2 =
\P_{\Theta_1}(n) + \P_{\Theta_1}(n+1). $ Analogously, for
$\Theta_2$ we can write $ \Delta \C (n) + 2 = \P_{\Theta_2}(n)   +
\P_{\Theta_2}(n+1)$ for all $ n\geq 1\,.$

 Let $N$ be the integer from
Corollary \ref{odjistehoNplati}. Adding the two previous equalities
and using Corollary \ref{odjistehoNplati}, we obtain  $ 2\Delta \C (n) + 4 \leq
\Delta \C (n) + 4$, i.e., $\Delta \C (n) = 0$ for $ n>N$ and thus $\uu$ is periodic.
\end{proof}

\begin{corollary}\label{komutuji}
Let $\Theta_1$ and  $\Theta_2$ be two distinct commuting
involutive antimorphisms such that their composition
$\Theta_1\Theta_2$ has no fixed letter, i.e., it is a derangement when restricted to $\A$.
If $\uu$ is an infinite word with language closed under $\Theta_1$ and
$\Theta_2$, then
$$
\Delta \C (n) + 4 \geq \P_{\Theta_1}(n) + \P_{\Theta_2}(n)  +
\P_{\Theta_1}(n+1) + \P_{\Theta_2}(n+1)\ \ \text{ for all } n \geq
1 \,.
$$
\end{corollary}

\begin{example} The reversal mapping  $\Tr$ on $\{0,1\}^*$ and the
antimorphism $\Theta$ determined by exchange of letters $0$ and $1$
used in Example \ref{dvaKratRich} satisfy the assumption of the
previous corollary. By an argument similar to the one used in the proof
of Corollary \ref{skorovzdy}, we deduce that $\Delta \mathcal{C}(n) = 0$
for all  $n\geq 1$, thus $ \mathcal{C}(2)=
\mathcal{C}(1)=\#\mathcal{A} = 2$. Therefore, the only infinite words
on the alphabet $\{0,1\}$  which are simultaneously $\Tr$-rich
and $\Theta$-rich are $\uu=(01)^\omega$ and
$\uu=(10)^\omega$\,.
\end{example}
The previous considerations justify the modification of the notion of $\Theta$-palindromic richness for words whose languages have
more symmetries. Before proceeding, we need to show that
there exist words for which the inequality  given in
Theorem \ref{nerovnostProDva} is in fact an equality. Let us show that
such a suitable example is the Thue-Morse word.\\

Let $\Theta$ be again the antimorphism on $\{0,1\}^*$ defined by
$\Theta(0) = 1$ and $\Theta(1) = 0$. This antimorphism commutes
with the reversal mapping $\Tr$ and $\Tr\Theta$ has no
fixed letter.  Consider the morphism  $0\mapsto
01$ and $1\mapsto 10$ which generates the Thue-Morse word. From
the form of the morphism one can easily see that the language of
its fixed points is invariant under $\Theta$ and $\Tr$. In
the article \cite{BlBrGaLa}, the classical palindromic complexity
and the $\Theta$-palindromic complexity of the Thue-Morse word is
described.

\begin{proposition}\label{TMpalindromy}
The palindromic complexity of the Thue-Morse word is
$$
\P_{\Tr}(n) = \left \{
\begin{array}{ll}
1 & \quad \text{ if } \ n  = 0, \\
2 & \quad \text{ if } \  1 \leq n \leq 4, \\
0 & \quad \text{ if } \ n \text{ is odd and } n \geq 5, \\
4 & \quad \text{ if } \ n \text{ is even and } 4^k < n \leq 3 \cdot 4^k, \text{ for } k \geq 1, \\
2 & \quad \text{ if } \ n \text{ is even and } 3 \cdot 4^k < n
\leq 4^{k+1}, \text{ for } k \geq 1.
\end{array}
\right.
$$
The $\Theta$-palindromic complexity of the Thue-Morse word is
$$
\P_{\Theta}(n) = \left \{
\begin{array}{ll}
1 & \quad \text{ if } \ n = 0, \\
2 & \quad\text{ if } \  n = 2, \\
0 & \quad \text{ if } \  n \text{ is odd and }, \\
4 & \quad \text{ if } \ n \text{ is even and } \frac{1}{2} \cdot 4^k < n \leq \frac{3}{2} \cdot 4^k, \text{ for } k \geq 1, \\
2 & \quad \text{ if } \ n \text{ is even and } \frac{3}{2} \cdot
4^k < n \leq \frac{1}{2} \cdot 4^{k+1}, \text{ for } k \geq 1.
\end{array}
\right.
$$
\end{proposition}
The factor complexity of the Thue-Morse word was described in
1989 independently in \cite{Br89} and \cite{LuVa}.
\begin{proposition}\label{TMcomplexita}
The first difference of factor complexity of the  Thue-Morse word is
$$
\Delta \C(n) = \left \{
\begin{array}{ll}
1 & \quad \text{ if } \ n = 0, \\
4 & \quad \text{ if } \ 2^k < n \leq 3 \cdot 2^{k-1}, \text{ for } k \geq 1, \\
2 & \quad \text{ otherwise}.
\end{array}
\right.
$$
\end{proposition}

Using these results on complexities we can show that the
Thue-Morse word is saturated by classical palindromes and
$\Theta$-palindromes up to the highest possible level given by the
inequality in Corollary \ref{komutuji}.

\begin{corollary}\label{TMrovnost} For the Thue-Morse word we have
$$
\Delta \C(n) + 4 = \P_{\Tr}(n) + \P_{\Tr}(n+1) +
\P_{\Theta}(n) + \P_{\Theta}(n+1)
$$
for all $n \geq 1$.
\end{corollary}

\begin{proof} The result follows immediately from Propositions \ref{TMcomplexita} and \ref{TMpalindromy}.
For reader's convenience we report the values of  $R(n) =
\P_{\Tr}(n) + \P_{\Tr}(n+1) + \P_{\Theta}(n) +
\P_{\Theta}(n+1)$ and $\Delta \C(n)$ in Table \ref{tab:TM_Rn}.

\renewcommand{\tabcolsep}{0.8cm}
\renewcommand{\arraystretch}{1.75}

\begin{table}
\begin{center}
\begin{tabular}{c|c|c}
$n$ & $R(n)$ & $\Delta \C(n)$ \\ \hline
$1$ & $2 + 2 + 0 + 2$ & $2$ \\ \hline
$2$ & $2 + 2 + 2 + 0$ & $2$ \\ \hline
$3$ & $2 + 2 + 0 + 4$ & $4$ \\ \hline
$4^k < n \leq \frac{3}{2} 4^k$ & $4+4$ & $4$ \\ \hline
$\frac{3}{2} 4^k < n \leq 2 \cdot 4^k$ & $4+2$ & $2$ \\ \hline
$2 \cdot 4^k < n \leq 3 \cdot 4^k$ & $4+4$ & $4$ \\ \hline
$3 \cdot 4^{k} < n \leq 4^{k+1}$ & $2+4$ & $2$ \\
\end{tabular}
\end{center}
\caption{Values $R(n)$ and $\Delta \C (n)$ for the Thue-Morse
word.} \label{tab:TM_Rn}
\end{table}

\end{proof}

In fact, the equality is also trivially attained for $n = 0$
while considering the inequality in its general form in Theorem \ref{nerovnostProDva}.

%


\section{Words with language closed under all elements of a group of
symmetries}\label{grupasymetrii}

If a finite set $G$ is a submonoid of $ AM({\A^*})$ such that its elements are non-erasing,
then, since it is finite, any of its elements restricted to the set of words
of length one is just a permutation on $\A$, and one can easily see that $G$ is a group.
If an antimorphism $\Theta$ is involutive, then
the corresponding permutation has cycles of length at most $2$. In
this section, we consider all antimorphisms with finite order, not
only of order $2$.

\begin{example} \label{Mahler}
The Champernowne word
$$ 12345678910111213141516171819202122232425262728293031\ldots$$
over the alphabet $\{0,1,2,\ldots, 9\}$ arises by writing
 decimal representations of all positive integers in increasing order.
The factor complexity of the Champernowne word is $\mathcal{C}(n) =
10^n$ and its language is invariant under any element of the group
$S_{10} \cup \Tr S_{10}$, where $S_{10}$ is the
group of permutations on a $10$-element set.

\end{example}

The Champernowne word has maximal factor complexity and its group of
symmetries is huge. The opposite extreme is a periodic word. We
shall see that its group of symmetries is much more restricted.

\begin{proposition} \label{periodicke_neni_uzavrene_na_neinvolutivni}
If $w\in \A^*$ is the shortest possible period of the periodic
word $\uu = w^{\omega}$ whose language  is closed under a non-erasing
antimorphism $\Theta: \A^* \mapsto \A^*$ of finite order, then
\begin{enumerate}
\item $\Theta$ is an involution;

\item $w=ps$,  where $p$ and $s$ are $\Theta$-palindromes;

\item $\P_\Theta(n)+ \P_\Theta(n+1) = 2$ for any $n\geq |w|$.

\end{enumerate}

\end{proposition}
\begin{proof}
According to the convention we made in Preliminaries,  $w$ contains
all letters from $\A$.
 If $\Theta(w) = w$, then $\Theta^2(w) = w$. As $\Theta$ is non-erasing and its order is finite, $\Theta^2$ is a morphism which just permutes the
letters. The equality $\Theta^2(w) = w$ implies that $\Theta^2 =\rm{Id}$ and we can put $p=w$ and $s= \varepsilon$.

If $\Theta(w)\neq w$, then  $\Theta(w)$ is a factor of the word
$ww$. Let $p$ and $s$ be factors such that $ww = p\Theta(w)s$.
It is easy to see that in fact $w = ps$ since $|w| = \Theta(w)$. Therefore, $ww =
p\Theta(s)\Theta(p)s$ and thus $w = p\Theta(s) = \Theta(p)s$. In
other words, $s$ and $p$ are $\Theta$-pa\-lindromes and both
 contain all letters of $\A$.  Analogously to the
previous case, this already implies involutivity of $\Theta$.

To show the last item, observe that any
$\Theta$-palindrome $q\in\L(\uu)$ of length at least $|w|$ is a
central factor of a palindrome $(sp)^ks(ps)^k$ or $(ps)^kp(sp)^k$
for some $k\in \mathbb{N}$. By minimality of the period
$w$, the central factors   of $(sp)^ks(ps)^k$ and $(ps)^kp(sp)^k$
of length $n\geq |w|$ are distinct.  If $|p|$ and
$|s|$ have opposite parities, then we have  one $\Theta$-palindrome of length $n$
and one $\Theta$-palindrome of length $n+1$. If $|p|$ and
$|s|$ have the same parities, then we have  two
$\Theta$-palindromes of length $n$ and none of length $n+1$
or vice versa.
\end{proof}

A crucial role for the newly proposed definition of richness with
respect to more antimorphisms is played by graphs of symmetries.
We can assign such a graph to any infinite word $\uu$ whose language is invariant
under a group $G\subset AM({\mathcal{A}^*})$. For the first time, the most
simple variant of this graph for $G=\{ \rm{Id}, \Tr\}$ appeared in
the proof of the main theorem of the article \cite{BaMaPe}.
See also \cite{BuLuGlZa}.
 In the previous section, we used this graph for $G=\{
\rm{Id}, \Theta_1,\Theta_2,\Theta_1\Theta_2 \}$. Both examples involve
simple groups containing only antimorphisms of order $2$. In fact, no
such restriction is necessary.

Let us consider a finite group $G\subset AM({\mathcal{A}^*})$. We
define a relation on ${\mathcal{A}^*}$ by
 $$ v\sim w \qquad \Longleftrightarrow \qquad v=\Theta(w) \ \ \hbox{for some } \ \Theta \in G. $$

It is obvious that $\sim$ is an equivalence relation and that any equivalence class, again denoted by $[w]$ for $w\in {\mathcal{A}^*}$,
has at most $\#G$ elements.

\begin{defi} Let $G\subset AM({\mathcal{A}^*})$ be a finite group,
 $\uu$ be an infinite word with language closed under each
$\Theta \in G$ and $n\in \mathbb{N}$.

\begin{description}
\item[1)] The directed graph  of symmetries of the word $\uu$ is
$\overrightarrow{\Gamma}_n(\uu) = (V,\overrightarrow{E})$ with the set
of vertices $$V=\{ [w]\,|\, w\in \L_n(\uu), w \ \hbox{is LS or RS}
\}$$ and  an edge $e\in \overrightarrow{E} \subset \L(\uu)$ starts
in a vertex $[w]$ and ends in a vertex $[v]$, if
\begin{itemize}
\item the prefix of $e$ of length $n$ belongs to $[w]$,
 \item the suffix of $e$ of length $n$ belongs to $[v]$,
\item $e$ has exactly two occurrences of special factors of length $n$.
\end{itemize}

\item[2)] The graph of symmetries of the word $\uu$ is
$\Gamma_n(\uu)=(V,E)$ with the same set of vertices as
$\overrightarrow{\Gamma}_n(\uu)$ and for any $e\in \L(\uu)$ we have
$$[e] \in E \qquad \Longleftrightarrow \qquad e \in \overrightarrow{E}.$$
\end{description}
\end{defi}

\newsavebox{\tempbox}

\begin{figure}[h!]%
\sbox{\tempbox}{\begin{minipage}[t]{0.35\textwidth}
\centering\includegraphics{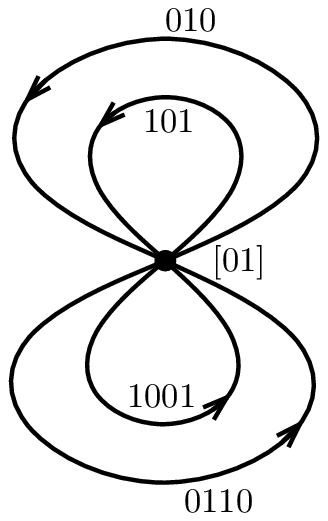}\vspace{\baselineskip}\end{minipage}}%
\centering
\subfloat[][]{ \label{fig:TM2_orientovany} \usebox{\tempbox}}%
\subfloat[][]{ \label{fig:TM2_neorientovany} %
\begin{minipage}[t]{0.35\textwidth}
\centering
\vbox to \ht\tempbox{%
\vfil
\includegraphics{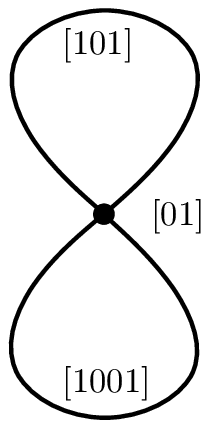}
\vfil}\vspace{\baselineskip}
\end{minipage}}%
\caption{(a) shows the graph $\protect \overrightarrow{\Gamma}_2(\uu_{TM})$ and (b) the graph $\Gamma_2(\uu_{TM})$  of the Thue-Morse word.}
\label{fig:TM2_grafy}
\end{figure}

Figure \ref{fig:TM2_grafy} shows these graphs for the Thue-Morse word.

\begin{defi} Let $G\subset AM({\mathcal{A}^*})$ be a finite group containing at least one antimorphism. We say
that an infinite  word $\uu$  is $G$-rich  (resp. almost
$G$-rich) if for each $n\in \mathbb{N}$ (resp. for each $n\in
\mathbb{N}$ up to finitely many exceptions) the following holds
\begin{itemize}
\item $\L(\uu)$ is invariant under all $\Theta \in G$;
  \item if $[e]$ is a loop in $\Gamma_n(\uu)$, then $e$ is a
  $\Theta$-palindrome for some $\Theta\in G$;
  \item the graph obtained from $\Gamma_n(\uu)$ by removing loops is a tree.
  \end{itemize}
\end{defi}

\begin{remark}  Let us compare the previous definition with the classical notions of  richness and almost richness.

\begin{itemize} \item Consider an eventually periodic word $\uu$ with language
closed under an antimorphism. Since  invariance  under an
antimorphism implies  recurrence, the word $\uu$ is periodic. Therefore for $n$ greater than the length of the period, the
graph $\overrightarrow{\Gamma}_n(\uu)$ is empty and thus a tree. It
means that, according to our definition, an eventually periodic word
is almost $G$-rich if and only if its language is closed under all
elements of the group $G$.

In \cite{BrHaNiRe},  the $\Tr$-defect of  periodic words is
studied. It is shown that $\Tr$-defect is finite  if and
only if the language of the  word is closed under $\Tr$. Let
us recall that words with finite $\Tr$-defect were
called  almost rich in \cite{GlJuWiZa}. Therefore, our definition of almost $G$-rich
periodic words does not contradict the old one in the case $G=\{\rm{Id}, \Tr\}$.

\item  Even for aperiodic words  our definition applied to the
group $G = \{\rm{Id}, \Tr\}$ is equivalent to the classical
definition of richness and almost richness on the set of words
with language closed under reversal, see Theorem 1.1, Proposition
1.2 and 3.5 in \cite{BuLuGlZa}. The same is valid for the group $G
= \{\rm{Id}, \Theta\}$, where $\Theta$ is an involutive antimorphism,
see Theorem 2 and Corollary 7 in \cite{Sta2010}. \end{itemize}
\end{remark}

Although we   allowed the group  $G$ to have antimorphisms of
higher order, in fact only the involutive antimorphisms $\Theta$ can
have a fixed point $w$ containing all letters from the alphabet.
Therefore, the notion of $\Theta$-palindromic complexity
$\P_{\Theta}$ makes sense only for  involutions.
Let $G^{(2)}$ be the set $G^{(2)} = \left \{ \Theta \in G \mid \Theta \text{ is an antimorphism and } \Theta^2 = \rm{Id} \right \}$, i.e., the set of involutive antimorphisms of $G$.

\begin{thm}\label{nerovnostProVice} Let $G\subset AM({\mathcal{A}^*})$ be a finite group containing an antimorphism and let
$\uu$ be an infinite  word whose language is invariant under all
elements of $G$. If there exists an integer $N\in \mathbb{N}$ such
that in any factor of $\uu$ of length at least $N$  all letters
of $\mathcal{A}$ occur, then
\begin{equation}\label{mistoStromovi} \Delta \C (n) + \# G \ \ \geq  \sum_{\Theta \in G^{(2)}}\Bigl(\P_{\Theta}(n) + \P_{\Theta}(n+1)\Bigr) \qquad
\hbox{for any} \ n \geq N\,.
\end{equation}
\end{thm}

\begin{proof}

Let $\Psi$ be an antimorphism of $G$. The mapping $\Theta
\mapsto \Psi\Theta$ is a bijection on $G$, satisfying
$$\phi\in G \ \hbox{ is a morphism} \quad \Longleftrightarrow \quad
\Psi\phi\in G \ \ \hbox{ is an antimorphism}.$$ This means that
$G$ has an even number of elements, say $\#G=2k$.\\

Let us fix $n \geq N$. First, suppose  that the graph
$\overrightarrow{\Gamma}_n(\uu)$ is nonempty. \\

Since  each factor of $\uu$ longer than $N$ contains all letters,
 for any two antimorphisms $\Theta_1,\Theta_2$ of $G$ we have\,:

\begin{equation}\label{dlouheanti}
\Theta_1\neq \Theta_2 \Longrightarrow  \Theta_1(v) \neq \Theta_2
(v) \ \ \hbox{for any $v$ such that $|v|\geq N$.}
\end{equation}
And similarly, for any two morphisms $\varphi_1$ and $\varphi_2$
of $G$ we have
\begin{equation}\label{dlouheantianti}
\varphi_1\neq \varphi_2 \Longrightarrow  \varphi_1(v) \neq
\varphi_2 (v) \ \ \hbox{for any $v$ such that $|v|\geq N$.}
\end{equation}
On the other hand, if $v$ is a $\Theta$-palindrome for an
antimorphism $\Theta\in G$, then for any antimorphism $\Theta_i
\in G$,  the word $\Theta_i(v)$ is a $(\Theta_i\Theta\Theta_i^{-1})$-palindrome. We may conclude the following for the
directed graph $\overrightarrow{\Gamma}_n(\uu)$ with $n \geq N$.

\begin{itemize}
\item A vertex $[w]$ has $\# G =2k$ elements if $w$ is not a
$\Theta$-palindrome for any antimorphism  $\Theta\in G$. Denote
the number of such vertices by $\beta$.

 \item  A vertex $[w]$ has $k$ elements if there exists an antimorphism $\Theta \in
 G$ such that $w = \Theta(w)$.  Denote
the number of such vertices by $\alpha$.

 \item If two distinct vertices $[w]$ and $[v]$ are connected by
 an edge $e$ starting in $w$ and ending in $v$, then there exist
 at least $2k$ edges between these two vertices, namely $k$ edges $\varphi(e)$ having
 the same orientation as $e$ and $k$ edges $\Theta(e)$ having the opposite
 orientation,
 for any morphism $\varphi\in G$ and any antimorphism $\Theta\in
 G$.

 \item No factor of length at least $n$  is  a $\Theta$-palindrome simultaneously
for two distinct antimorphisms $\Theta$.
\end{itemize}
Using the property mentioned in the last item, the proof is
now easier than the proof of Theorem \ref{nerovnostProDva}.

As all special factors of $ \L_n(\uu)$ belong to the classes
forming vertices in $\overrightarrow{\Gamma}_n(\uu)$, any factor $f\in
\L_{n+1}(\uu)$  as well as any non-special factor in $f\in \L_{n}(\uu)$
is a factor (which is neither  the prefix of length $n$, nor
a suffix of length $n$) of exactly one edge $e$ in
$\overrightarrow{\Gamma}_n(\uu)$. Clearly $\Theta(\overrightarrow{E}) =
\overrightarrow{E}$. If such a factor $f$ is a
$\Theta$-palindrome, then necessarily $\Theta(e)=e$ and $e$ is a
loop in $\overrightarrow{\Gamma}_n(\uu)$. For the number of edges in the
directed graph $\overrightarrow{\Gamma}_n(\uu)$ we have the lower bound:
\begin{equation}\label{pocetHran1} \#\overrightarrow{E} \geq 2k(\alpha + \beta -1) + A
\end{equation}
 where
\begin{equation}\label{pocetHran2}A \geq \sum_{\Theta \in G^{(2)}}\Bigl(\P_{\Theta}(n) +
\P_{\Theta}(n+1)\Bigr) - k\alpha\,.
\end{equation}
Again, as in the proof of Theorem \ref{nerovnostProDva}, the
number of edges  $\#\overrightarrow{E}$ can be written as
$$
\#\overrightarrow{ E}  = \sum_{w\in \L_n(\uu), w \text{ special}} \# \Lext (w)
=
\sum_{\substack{w\in \L_n(\uu), w \text{ special} \\ \# [w] = 2k }} \# \Lext
(w) + \sum_{\substack{w\in \L_n(\uu), w \text{ special} \\ \# [w] = k  }} \#
\Lext (w). $$ Since the first sum on the right side has $2k\beta$
summands and the second sum has $k\alpha$ summands, we obtain
$$
\#\overrightarrow{ E} = \sum_{w\in \L_n(\uu), w \text{ special}} \Bigl(\# \Lext
(w) - 1\Bigr) + 2k \beta + k\alpha =  \Delta \C(n)   + 2k \beta +
k\alpha.
$$
This together with \eqref{pocetHran1} and \eqref{pocetHran2}
implies the theorem in case that $\overrightarrow{\Gamma}_n(\uu)$ is
nonempty.

If $\overrightarrow{\Gamma}_n(\uu)$ is empty, then  $\uu$ is periodic
and, according to the point 1. in Proposition \ref{periodicke_neni_uzavrene_na_neinvolutivni}, the left side of
the inequality \eqref{mistoStromovi} is twice the number of
involutive antimorphisms in $G$. According to the point 3. in
Proposition \ref{periodicke_neni_uzavrene_na_neinvolutivni}, the
right  side of the inequality has the same value.
\end{proof}

\begin{remark}\label{Test} The previous proof enables us to test $G$-richness by verifying the equality
in \eqref{mistoStromovi} instead of looking at the graph
$\Gamma_n$  for $n\geq N$. For $n <N$ we still have to check that
$\Gamma_n$ is a tree.
\end{remark}

\begin{remark} \label{rem:weaker}
The assumption on the integer $N$ in Theorem \ref{nerovnostProVice}
 can be replaced by the following weaker assumption:
there exists an integer $N$ such that
for any two antimorphisms $\Theta_1, \Theta_2 \in G$ it holds
$$
\Theta_1\neq \Theta_2 \Longrightarrow  \Theta_1(v) \neq \Theta_2
(v) \ \ \hbox{for any $v$ with $|v|\geq N$,}
$$
and for any two morphisms $\varphi_1, \varphi_2 \in G$ it holds
$$
\varphi_1\neq \varphi_2 \Longrightarrow  \varphi_1(v) \neq
\varphi_2 (v) \ \ \hbox{for any $v$ with $|v|\geq N$.}
$$
\end{remark}

Since the existence of $N$ required in Theorem
\ref{nerovnostProVice} is trivially satisfied  for uniformly
recurrent words, the proof of Theorem \ref{nerovnostProVice} gives
a  simple criterion for almost $G$-richness.

\begin{corollary}\label{almostG-Rich}
 Let $G\subset AM({\mathcal{A}^*})$ be a finite group containing at least one antimorphism.
If an infinite uniformly recurrent  word $\uu$ has its language
invariant under all elements of $G$, then  $\uu$ is almost
$G$-rich if and only if there exists  $N\in \mathbb{N}$ such that
 $$
\Delta \C (n) + \# G \ \ =  \sum_{\Theta \in G^{(2)}}\Bigl(\P_{\Theta}(n) +
\P_{\Theta}(n+1)\Bigr) \qquad \hbox{for any} \ n \geq N\,. $$
\end{corollary}

\section{Examples of $G$-rich words} \label{sec:examples}

\subsection{The Thue-Morse word and its generalizations}
To demonstrate  $G$-richness of the Thue-Morse word, we use Remark
\ref{Test}.  From Corollary \ref{TMrovnost} and the shape of
$\Gamma_2(\uu_{TM})$ in Figure \ref{fig:TM2_neorientovany} we can deduce the
following statement.

\begin{corollary} The Thue-Morse word is $G$-rich,
where $G$ is the group generated by the reversal mapping and  the
involutive antimorphism determined by the exchange of letters.
\end{corollary}

The following generalization of the Thue-Morse word was considered
for instance in \cite{AlSh}. Let $s_b(n)$ denote the sum of digits
in the base-$b$ representation of the integer $n$. The infinite
word $\tt_{b,m}$ is defined as
$$
\tt_{b,m} = \left ( s_b(n) \mod m \right )_{n=0}^{+\infty}.
$$
Using this notation, the famous Thue-Morse word equals
$\tt_{2,2}$. The word $\tt_{b,m}$ is over the alphabet
$\{0,1,\ldots, m-1\}$ and is also a fixed point of a  primitive
morphism, as already mentioned in \cite{AlSh}. To abbreviate the formulas, let us denote $i\oplus_m
j = i+j \mod m$ for any $i,j \in \{0,1,\ldots, m-1\}$. It is easy
to see  that the morphism defined by
$$ \varphi(k) =  k(k\oplus_m 1)(k\oplus_m 2)\ldots \bigl(k\oplus_m
(b-1)\bigr)\quad \text{for any} \ k \in \{0,1,\ldots, m-1\}
$$
fixes the word $\tt_{b,m}$.

\begin{example} For  parameters $(b,m) \neq (2,2)$ an explicit description of the group $G$ under which
the language of  $\tt_{b,m}$ is invariant
 and  values of palindromic complexities $\mathcal{P}_\Theta$ is not available.
A known fact which can be easily seen from the morphism $\varphi$ is that
if $m \mid (b-1)$, then $\tt_{b,m}$ is periodic.
We first present some periodic examples.

 \begin{description}

\item[The periodic case $m \mid (b-1)$] \quad In this case, one can see that from the morphism that
$\tt_{b,m} = (01 \ldots (m-1))^{\omega}$.
 For example, if $b=5$  and $m=2$, then the morphism has the form $0 \mapsto
01010$ and $1 \mapsto 10101$.

For $m=2$, the word $(01)^{\omega}$  trivially satisfies the equality in
Corollary \ref{komutuji} for the reversal mapping and the antimorphism
determined by the exchange of  0 and 1. Thus $\tt_{2k+1,2}$ is
$G$-rich with the same group $G$ as the Thue-Morse word
$\tt_{2,2}$.

For $m = 3$, the language of the word $(012)^{\omega}$ is closed under all elements of  a group generated by the $3$
involutive antimorphisms on $\{0,1,2\}^*$ different from the
reversal mapping. The proof of richness is left to the reader.

\end{description}

For other values of $b$ and $m$, we used computer resources, namely the open-source mathematical software \texttt{Sage} \cite{sage_4_6}, to look for candidates exhibiting richness.

\begin{description}

 \item[The word $\tt_{2k,2}$] \quad From the  shape of the morphism it follows that the language of the language of the word  $\tt_{2k,2}$  is invariant under the same group $G$ as
  $\uu_{TM}$. For these words  our computer experiments suggest that equality also holds. Table \ref{tab:TM_42} shows some values for $\Delta
\C(n)$ and $\displaystyle R(n) =  \sum_{\Theta \in G^{(2)}}\Bigl(\P_{\Theta}(n) +
\P_{\Theta}(n+1)\Bigr)$ for $\tt_{4,2}$.

\renewcommand{\tabcolsep}{1cm}
\renewcommand{\arraystretch}{2}
\begin{table}
\begin{center}
\begin{tabular}{c|c|c}
$n$ & $R(n)$ & $\Delta \C(n)$ \\ \hline
$0 < n < 17$ & $6$ & $2$ \\ \hline $17 \leq n < 29$ & $8$ & $4$ \\
\hline $29 \leq n < 65$ & $6$ & $2$ \\ \hline $65 \leq n < 113$ &
$8$ & $4$ \\ \hline
$113$ & $6$ & $2$ \\
\end{tabular}
\end{center}
\caption{Values $R(n)$ and $\Delta \C (n)$ for $\tt_{4,2}$, $\#G = 4$.}
\label{tab:TM_42}
\end{table}

\item[The word $\tt_{2,4}$] \quad This word is a fixed point of
the morphism
$$
0\mapsto 01, \ \ 1\mapsto 12, \ \ 2\mapsto 23,\  \  \text{and} \ \
3\mapsto 30.
$$
A computer test on factors of length $100$ of the prefix of $\tt_{2,4}$ of length $30000$
suggests  that its language is invariant under four antimorphisms:

$\Theta_1: \ 0\to 0, 1\to 3, 2\to 2, 3\to 1$,

$\Theta_2: \ 0\to 1, 1\to 0, 2\to 3, 3\to 2$,

$\Theta_3: \ 0\to 2, 1\to 1, 2\to 0, 3\to 3$,

$\Theta_4: \ 0\to 3, 1\to 2, 2\to 1, 3\to 0$.

\noindent The group $G$ generated by those $4$ antimorphisms has  8 elements.
It seems that the word $\tt_{2,4}$ is $G$-rich as well.

\end{description}
Our only conclusion is that identification of $G$-rich words among
$\tt_{b,m}$ requires a further  study.

\end{example}

All examples of words we proved to be $G$-rich or we suspect
to be $G$-rich have a common property: any antimorphism in $G$ has
order two. Proposition
\ref{periodicke_neni_uzavrene_na_neinvolutivni} says that   for
periodic words this property is necessary.

\subsection{A $G$-rich word with irrational densities of letters}

If the language of an infinite word with well-defined densities of letters is
invariant under an antimorphism $\Theta$ of finite order and
$\Theta$ is not the reversal mapping $\Tr$, then the density vector is invariant under the permutation corresponding to
 $\Theta$. For example, the densities of both letters $1$ and $0$
in the Thue-Morse word are  necessarily $\tfrac{1}{2}$, the
densities of all letters of the Champernowne word from Example
\ref{Mahler} are $\frac{1}{10}$. Nevertheless, the densities of
letters  need not be rational.

In this section we describe a $G$-rich word with irrational
densities of letters.

\begin{example} \label{prikladG2}
Let $\varphi$ be the morphism on $\{0,1,2,3\}^*$ defined as
\renewcommand\tabcolsep{6pt}
\renewcommand\arraystretch{1}
$$
\varphi: \left \{
\begin{array}{l}
0 \mapsto 0130 \\
1 \mapsto 1021\\
2 \mapsto 102 \\
3 \mapsto 013
\end{array}
\right.
$$
and let $\uu$ be a fixed point of $\varphi$. The matrix of this
morphism $M_\varphi$ and the eigenvector $x_\Lambda$ corresponding
to the dominant eigenvalue $\Lambda = 2+\sqrt{3}$ are
$$ M_\varphi = \left(
\begin{array}{cccc}
2&1&1&1\\
1&2&1&1\\
0&1&1&0\\
1&0&0&1
\end{array}\right) \quad {\rm and} \quad  x_\Lambda = \left(
\begin{array}{c}
\tfrac{\sqrt{3}-1}{2}\\
\tfrac{\sqrt{3}-1}{2}\\
\tfrac{2-\sqrt{3}}{2}\\
\tfrac{2-\sqrt{3}}{2}
\end{array}\right) \,.
$$
The matrix $M_\varphi$ is primitive and  therefore our
morphism is  primitive as well (see \cite{Fogg}). It is known that the components of the
eigenvector corresponding to the dominant eigenvalue are
proportional to the densities of letters. In our case, the letters
$0$ and $1$ have density $\tfrac{\sqrt{3}-1}{2}$. The letters
$2$ and $3$ have density $\tfrac{2-\sqrt{3}}{2}$.\\

We show the following properties of the word $\uu$:

\begin{enumerate}
\item
  Language $\Lu$ is closed under two involutive antimorphisms $\Ta$ and $\Tb$,
where
$$
\Ta: 0 \mapsto 1, 1 \mapsto 0, 2 \mapsto 2, 3 \mapsto 3 \quad
\text{and} \quad \Tb: 0 \mapsto 0, 1 \mapsto 1, 2 \mapsto 3, 3
\mapsto 2\,.
$$

\item  The first increment of factor complexity satisfies
$$\Delta \mathcal{C}(n)= 2\ \ \text{ for any}\ \  n\in \N^+\,.$$

 \item Only $\Theta_1$-palindromes of length $1$ and $2$  occurring in $\L(\uu)$  are factors  $2,3, 10$ and
$01$; only $\Theta_2$-palindromes of length $1$ and $2$  are
factors $0$ and $1$.

\item Any $\Theta_i$-palindrome $w \in \L(\uu)$ has a unique
$\Theta_i$-palindromic extension, i.e., there exists a unique letter
$a\in \mathcal{A}$ such that $aw\Theta_i(a) \in  \L(\uu)$.

\end{enumerate}

\end{example}

\begin{proof}[Proof of the properties of $\uu$ defined in Example \ref{prikladG2}]  $\,$

\begin{description}

\item[Property 3.]  To check this property is an easy task since
$\L_2(\uu) = \{ 02, 21, 13, 30, 01, 10 \}$.

\item[Property 1.]    As $\uu$ is a fixed point of the primitive
morphism $\varphi$, the word $\uu$ is uniformly recurrent.
To prove the invariance of $\L(\uu)$ under $\Theta_i$, it is
sufficient to show that $\uu$ contains infinitely many
$\Theta_i$-palindromes. We give a construction producing from
a $\Theta_i$-palindrome a longer $\Theta_i$-palindrome.

 Let $a \in \A$. Denote by $p_a$
the unique factor of $\uu$ of length $5$ such that $\varphi(a)p_a$
is a factor of $\uu$. The correctness of this definition can be
seen by looking at the images by $\varphi$
 of the factors of $\uu$ of length $2$, as listed above.
One can show that $p_0 = p_2 = 10210$ and $p_1 = p_3 = 01301$.

Fix $i \in \{1,2\}$.
Let $ab \in \L_2(\uu)$.
It is easy to show that
\begin{equation} \label{prikladG2_lemma}
\Theta_i (p_a) = p_{\Theta_i(b)}.
\end{equation}

We now prove the following claim.
If $w = w_1 \ldots w_n \in \Lu$, then
\begin{equation} \label{prikladG2_claim}
\Theta_i \big( \varphi(w)p_{w_n}  \big) = \varphi \big(
\Theta_i(w) \big) p_{\Theta_i(w_1)}.
\end{equation}
The claim can be shown by induction on $|w|$. Indeed, one can easily verify that for $a \in \A$ we have
\begin{equation}\label{prikladG2_indukcestart}
\Theta_i \big(\varphi(a)p_a\big)  =\varphi \big(\Theta_i(a)\big) p_{\Theta_i(a)}.
\end{equation}
Suppose now the claim holds for the word $w = w_1 \ldots w_n \in \Lu$.
We show that it holds also for $aw \in \Lu$, $a \in \A$.
We have
$$
S=\Theta_i \big(\varphi(aw)p_{w_n} \big) =
\Theta_i \big(\varphi(w)p_{w_n} \big)\Theta_i \big(\varphi(a)\big) =
\varphi \big( \Theta_i(w) \big)
p_{\Theta_i(w_1)}\Theta_i \big(\varphi(a)\big).
$$
Since $aw_1 \in \Lu$, using \eqref{prikladG2_lemma} we have
$$
 p_{\Theta_i(w_1)}  = \Theta_i \left(p_a\right)\,.
$$
We  continue to manipulate the equation using \eqref{prikladG2_indukcestart}
  $$S= \varphi \big( \Theta_i(w)
\big)\Theta_i \left(p_a\right)\Theta_i \big(\varphi(a) \big) = \varphi
\big( \Theta_i(w) \big)\Theta_i \big(\varphi(a)p_a \big)
=\varphi \big( \Theta_i(w) \big)
\varphi \big(\Theta_i(a) \big) p_{\Theta_i(a)}
$$
and finally
$$
S = \varphi \bigl( \Theta_i(w) \Theta_i(a)\bigr)
p_{\Theta_i(a)}=\varphi \bigl( \Theta_i(aw) \bigr) p_{\Theta_i(a)}
$$
as we claimed.

Let $w$ be a $\Theta_i$-palindrome.
Using \eqref{prikladG2_claim} we have
$$
\Theta_i \big( \varphi(w)p_{w_n} \big) = \varphi \big( \Theta_i(w) \big)
p_{\Theta_i(w_1)}  =  \varphi( w ) p_{w_n} .
$$
Therefore, $\varphi( w ) p_{w_n}$ is a $\Theta_i$-palindrome as
well.

\item[Property 2.]  We evaluate $\Delta \C(n)$. As we have
explained in the Preliminaries, we need to look at left special
factors and their extensions. In our word $\uu$, there are only
two left special factors of length one:  the letter $0$ with
$\Lext(0) = \{1,3\}$ and the letter $1$  with  $\Lext(1) = \{0,2\}$.
From the shape of the morphism we see that for any LS factor
$w$ its image $\varphi(w)$ is LS as well, and moreover,  $\Lext(w)=
\Lext(\varphi(w))$. Thus the factors $\varphi^k(0)$ and
$\varphi^k(1)$ are LS factors both with two left extensions for
any $k \in \mathbb{N}$. Since any prefix of LS factor is a LS
factor as well, we can deduce, that any prefix of a fixed point
$\lim\limits_{k\to \infty}\varphi^k(0)$ or $\lim\limits_{k\to
\infty}\varphi^k(1)$  is a left special factor.

For any length $n$ we  have two LS factors of length $n$ each with
two left extensions. To show that there are no other LS factors in $\uu$, one has to show that any left special factor $w$
longer than one is a prefix of $\varphi(v)$ where $v$ is
 a LS factor with the same left extension. A proof of this part is left to the reader.

Using the equation \eqref{DeltaC} we  can conclude that $\Delta
\C(n) = 2$ for all $n \geq 1$.

\item[Property 4.] The invariance of $\uu$ under $\Theta_i$
implies  that $\Theta_i(\Lext(v)) = \Rext(\Theta_i(v))$ for any $v
\in \L(\uu)$.  Analogously,  for bilateral orders we have $\b(v) =
\b(\Theta_i(v))$.

Thus, our description of LS  factors gives immediately that for any
$n$ there exist in $\uu$ exactly two RS factors of length $n$, each
with two extensions. Moreover, $v$ is LS if and only if  $\Theta_1(v)$ is RS
if and only if $\Theta_2\Theta_1(v)$ is LS.  Therefore, two LS  factors form a
pair $v$ and $\Theta_2\Theta_1(v)$ and have  the same bilateral
order.

We show by contradiction that any $\Theta_1$-palindrome has a
unique $\Ta$-palindromic extension. Consider a factor  $w =
\Theta_1(w)\in \L(\uu)$.

First, suppose that $w$ has no $\Theta_1$-palindromic extension.
Since any factor has at least one left  and one right extension,
there exist letters $a$  and $b$ such that $awb \in \L(\uu)$. As $w$ has
no $\Theta_1$-palindromic extensions, we have $b\neq \Theta_1(a)$.
The invariance of language under $\Theta_1$  implies that
$\Theta_1(b)w\Theta_1(a) \in  \L(\uu)$.   It means that  $w$ is a
bispecial factor with the left extensions $a, \Theta_1(b)$ only
and with the right extensions $b, \Theta_1(a)$ only. Thus $\b(w) =
-1$. As we have mentioned, the second BS factor of the same length
is the factor $\Theta_2(w)$ and for its bilateral order we have
$\b(\Theta_2(w))=\b(w) = -1$. According to \eqref{DeltaNaDruhuC} we
have  $\Delta^2 \C (n) = \sum_{w \in \L_n(\uu)} \b(w) = -2$ which is a
contradiction with Property 2.

Now suppose that $w$ has two  $\Theta_1$-palindromic extensions,
i.e., there exist  two letters $a\neq b$, such that
$aw\Theta_1(a)$  and $bw\Theta_1(b)$ belong to $\Lu$.
We have either $\b(w) = -1$ or $\b(w) = 1$ and we can repeat the argument
to deduce  a contradiction $\Delta^2 \C (n) = \pm 2$.

\end{description}
\end{proof}

Finally,  Property 3  implies  $P_{\Ta}(1) + P_{\Ta}(2) +
P_{\Tb}(1) + P_{\Tb}(2) = 6$.   Thus by using Property 4, we have
for all $n \geq 1$
$$
P_{\Ta}(n) + P_{\Ta}(n + 1) + P_{\Tb}(n) + P_{\Tb}(n + 1) = 6.
$$
Property 1 means $\# G =4$.   Together with $\Delta \C (n)=2$
from Property 2 it gives for any $n \geq 1$ the equality
$$\Delta \C (n) + \# G  =P_{\Ta}(n) + P_{\Ta}(n + 1) + P_{\Tb}(n) + P_{\Tb}(n +
1)$$ and thus the $G$-richness of $\uu$ (using Remark \ref{rem:weaker}).

\section{Open problems}
Several papers were devoted to rich words, stating many results.
How fruitful is
 the new definition of the $G$-richness remains open.
Answers to questions listed below may help to clarify that.

\begin{enumerate}

\item \label{oo1} If  the group $G$ only contains the identity and an
involutive antimorphism $\Theta$ , then  the $G$-rich words (i.e.
$\Theta$-rich words)  can be characterized by return words, see
\cite{GlJuWiZa}, \cite{Sta2010}. Is there such a
characterization of $G$-rich words for general $G$?

\item Palindromic closure ($\Theta$-palindromic closure) is a
very effective tool for constructing rich and almost rich
words, see for example \cite{BuLuLuZa2} and \cite{GlJuWiZa}.
 Can  $\Theta$-palindromic closures be used to construct
$G$-rich words?

\item  \label{oo3} Can a reasonable  $G$-analogue to $\Theta$-defect be
defined for a group of symmetries $G$?

\item How to modify the right side of the inequality in Theorem \ref{nerovnostProVice} to obtain an inequality for all $n \in
\mathbb{N}$ as it holds in Theorem \ref{nerovnostProDva}?

\item Is there an inequality analogous to \eqref{ABCD} if one
replaces the classical palindromic complexity $\mathcal{P}$ by the
$\Theta$-palindromic complexity $\mathcal{P}_\Theta$ for some antimorphism $\Theta$? Can  this
 inequality be improved if  $\uu$ contains simultaneously infinitely many
palindromes and $\Theta$-palindromes?

\item Given a group $G \subset AM(\A^*)$, how to find an infinite word $\uu$ such that it is (almost) $G$-rich?

\item Is there an explicit formula for the sequence $r(n)$ from Remark \ref{rich_long}?

\end{enumerate}

During the review process of this article, we answered questions \ref{oo1} and \ref{oo3} in \cite{PeSta2}.

\section*{Acknowledgments}
We would like to thank the anonymous referees for their numerous suggestions that improved the presentation of this paper.
The experimental results were done using the open-source mathematical software \texttt{Sage} \cite{sage_4_6}.
This work was supported by the Czech Science Foundation
grant GA\v CR 201/09/0584, by the grants MSM6840770039 and LC06002
of the Ministry of Education, Youth, and Sports of the Czech
Republic, and by the grant of the Grant Agency of the Czech
Technical University in Prague grant No. SGS11/162/OHK4/3T/14.


\begin{thebibliography}{22}
\expandafter\ifx\csname natexlab\endcsname\relax\def\natexlab#1{#1}\fi
\providecommand{\bibinfo}[2]{#2}
\ifx\xfnm\relax \def\xfnm[#1]{\unskip,\space#1}\fi
\bibitem[{Allouche et~al.(2003)Allouche, Baake, Cassaigne and
  Damanik}]{AlBaCaDa}
\bibinfo{author}{J.P. Allouche}, \bibinfo{author}{M.~Baake},
  \bibinfo{author}{J.~Cassaigne}, \bibinfo{author}{D.~Damanik},
  \bibinfo{title}{Palindrome complexity}, \bibinfo{journal}{Theoret. Comput.
  Sci.} \bibinfo{volume}{292} (\bibinfo{year}{2003}) \bibinfo{pages}{9--31}.
\bibitem[{Allouche and Shallit(2000)}]{AlSh}
\bibinfo{author}{J.P. Allouche}, \bibinfo{author}{J.~Shallit},
  \bibinfo{title}{Sums of digits, overlaps, and palindromes},
  \bibinfo{journal}{Discrete Math. Theoret. Comput. Sci.} \bibinfo{volume}{4}
  (\bibinfo{year}{2000}) \bibinfo{pages}{1--10}.
\bibitem[{Bal{\'a}{\v z}i et~al.(2007)Bal{\'a}{\v z}i, Mas{\'a}kov{\'a} and
  Pelantov{\'a}}]{BaMaPe}
\bibinfo{author}{P.~Bal{\'a}{\v z}i}, \bibinfo{author}{Z.~Mas{\'a}kov{\'a}},
  \bibinfo{author}{E.~Pelantov{\'a}}, \bibinfo{title}{Factor versus palindromic
  complexity of uniformly recurrent infinite words}, \bibinfo{journal}{Theoret.
  Comput. Sci.} \bibinfo{volume}{380} (\bibinfo{year}{2007})
  \bibinfo{pages}{266--275}.
\bibitem[{Balkov\'a et~al.(2011{\natexlab{a}})Balkov\'a, Pelantov\'a and
  Starosta}]{BaPeSta3}
\bibinfo{author}{L.~Balkov\'a}, \bibinfo{author}{E.~Pelantov\'a},
  \bibinfo{author}{{\v{S}}.~Starosta}, \bibinfo{title}{Infinite words with
  finite defect}, \bibinfo{journal}{Adv. in Appl. Math.} \bibinfo{volume}{47}
  (\bibinfo{year}{2011}{\natexlab{a}}) \bibinfo{pages}{562--574}.
\bibitem[{Balkov\'a et~al.(2011{\natexlab{b}})Balkov\'a, Pelantov\'a and
  Starosta}]{BaPeSta4}
\bibinfo{author}{L.~Balkov\'a}, \bibinfo{author}{E.~Pelantov\'a},
  \bibinfo{author}{{\v{S}}.~Starosta}, \bibinfo{title}{{On {B}rlek-{R}eutenauer
  conjecture}}, \bibinfo{journal}{Theoret. Comput. Sci.} \bibinfo{volume}{412}
  (\bibinfo{year}{2011}{\natexlab{b}}) \bibinfo{pages}{5649--5655}.
\bibitem[{{Blondin Mass{\'e}} et~al.(2008){Blondin Mass{\'e}}, Brlek, Garon and
  Labb{\'e}}]{BlBrGaLa}
\bibinfo{author}{A.~{Blondin Mass{\'e}}}, \bibinfo{author}{S.~Brlek},
  \bibinfo{author}{A.~Garon}, \bibinfo{author}{S.~Labb{\'e}},
  \bibinfo{title}{Combinatorial properties of \mbox{$f$-palindromes} in the
  {T}hue-{M}orse sequence}, \bibinfo{journal}{Pure Math. Appl.}
  \bibinfo{volume}{19} (\bibinfo{year}{2008}) \bibinfo{pages}{39--52}.
\bibitem[{Brlek(1989)}]{Br89}
\bibinfo{author}{S.~Brlek}, \bibinfo{title}{Enumeration of factors in the
  {T}hue-{M}orse word}, \bibinfo{journal}{Discrete Appl. Math.}
  \bibinfo{volume}{24} (\bibinfo{year}{1989}) \bibinfo{pages}{83--96}.
\bibitem[{Brlek et~al.(2004)Brlek, Hamel, Nivat and Reutenauer}]{BrHaNiRe}
\bibinfo{author}{S.~Brlek}, \bibinfo{author}{S.~Hamel},
  \bibinfo{author}{M.~Nivat}, \bibinfo{author}{C.~Reutenauer},
  \bibinfo{title}{On the palindromic complexity of infinite words},
  \bibinfo{journal}{Internat. J. Found. Comput.} \bibinfo{volume}{15}
  (\bibinfo{year}{2004}) \bibinfo{pages}{293--306}.
\bibitem[{Brlek and Reutenauer(2011)}]{BrRe-conjecture}
\bibinfo{author}{S.~Brlek}, \bibinfo{author}{C.~Reutenauer},
  \bibinfo{title}{Complexity and palindromic defect of infinite words},
  \bibinfo{journal}{Theoret. Comput. Sci.} \bibinfo{volume}{412}
  (\bibinfo{year}{2011}) \bibinfo{pages}{493--497}.
\bibitem[{Bucci et~al.(2009)Bucci, {De Luca}, Glen and Zamboni}]{BuLuGlZa}
\bibinfo{author}{M.~Bucci}, \bibinfo{author}{A.~{De Luca}},
  \bibinfo{author}{A.~Glen}, \bibinfo{author}{L.Q. Zamboni}, \bibinfo{title}{A
  connection between palindromic and factor complexity using return words},
  \bibinfo{journal}{Adv. in Appl. Math.} \bibinfo{volume}{42}
  (\bibinfo{year}{2009}) \bibinfo{pages}{60--74}.
\bibitem[{Bucci et~al.(2008)Bucci, de~Luca, {De Luca} and Zamboni}]{BuLuLuZa2}
\bibinfo{author}{M.~Bucci}, \bibinfo{author}{A.~de~Luca},
  \bibinfo{author}{A.~{De Luca}}, \bibinfo{author}{L.Q. Zamboni},
  \bibinfo{title}{On different generalizations of episturmian words},
  \bibinfo{journal}{Theoret. Comput. Sci.} \bibinfo{volume}{393}
  (\bibinfo{year}{2008}) \bibinfo{pages}{23--36}.
\bibitem[{Cassaigne(1997)}]{Ca}
\bibinfo{author}{J.~Cassaigne}, \bibinfo{title}{Complexity and special
  factors}, \bibinfo{journal}{Bull. Belg. Math. Soc. Simon Stevin 4}
  \bibinfo{volume}{1} (\bibinfo{year}{1997}) \bibinfo{pages}{67--88}.
\bibitem[{Damanik and Zare(2000)}]{DaZa}
\bibinfo{author}{D.~Damanik}, \bibinfo{author}{D.~Zare},
  \bibinfo{title}{Palindrome complexity bounds for primitive substitution
  sequences}, \bibinfo{journal}{Discrete Math.} \bibinfo{volume}{222}
  (\bibinfo{year}{2000}) \bibinfo{pages}{259--267}.
\bibitem[{Droubay et~al.(2001)Droubay, Justin and Pirillo}]{DrJuPi}
\bibinfo{author}{X.~Droubay}, \bibinfo{author}{J.~Justin},
  \bibinfo{author}{G.~Pirillo}, \bibinfo{title}{Episturmian words and some
  constructions of de {L}uca and {R}auzy}, \bibinfo{journal}{Theoret. Comput.
  Sci.} \bibinfo{volume}{255} (\bibinfo{year}{2001}) \bibinfo{pages}{539--553}.
\bibitem[{Fogg(2002)}]{Fogg}
\bibinfo{author}{N.P. Fogg}, \bibinfo{title}{Substitutions in Arithmetics,
  Dynamics and Combinatorics}, volume \bibinfo{volume}{1794} of
  \textit{\bibinfo{series}{Lecture notes in mathematics}},
  \bibinfo{publisher}{Springer}, \bibinfo{edition}{1st} edition,
  \bibinfo{year}{2002}.
\bibitem[{Glen et~al.(2009)Glen, Justin, Widmer and Zamboni}]{GlJuWiZa}
\bibinfo{author}{A.~Glen}, \bibinfo{author}{J.~Justin},
  \bibinfo{author}{S.~Widmer}, \bibinfo{author}{L.Q. Zamboni},
  \bibinfo{title}{Palindromic richness}, \bibinfo{journal}{European J. Combin.}
  \bibinfo{volume}{30} (\bibinfo{year}{2009}) \bibinfo{pages}{510--531}.
\bibitem[{Hof et~al.(1995)Hof, Knill and Simon}]{HoKnSi}
\bibinfo{author}{A.~Hof}, \bibinfo{author}{O.~Knill},
  \bibinfo{author}{B.~Simon}, \bibinfo{title}{Singular continuous spectrum for
  palindromic {S}chr\"{o}dinger operators}, \bibinfo{journal}{Comm. Math.
  Phys.} \bibinfo{volume}{174} (\bibinfo{year}{1995})
  \bibinfo{pages}{149--159}.
\bibitem[{de~Luca and Varricchio(1989)}]{LuVa}
\bibinfo{author}{A.~de~Luca}, \bibinfo{author}{S.~Varricchio},
  \bibinfo{title}{Some combinatorial properties of the {T}hue-{M}orse sequence
  and a problem in semigroups}, \bibinfo{journal}{Theoret. Comput. Sci.}
  \bibinfo{volume}{63} (\bibinfo{year}{1989}) \bibinfo{pages}{333--348}.
\bibitem[{Pelantov\'a and Starosta(2011{\natexlab{a}})}]{PeSta_Milano}
\bibinfo{author}{E.~Pelantov\'a}, \bibinfo{author}{{\v{S}}.~Starosta},
  \bibinfo{title}{Infinite words rich and almost rich in generalized
  palindromes}, in: \bibinfo{editor}{G.~Mauri}, \bibinfo{editor}{A.~Leporati}
  (Eds.), \bibinfo{booktitle}{Developments in Language Theory}, volume
  \bibinfo{volume}{6795} of \textit{\bibinfo{series}{Lecture Notes in Computer
  Science}}, \bibinfo{publisher}{Springer-Verlag, Berlin, Heidelberg},
  \bibinfo{year}{2011}{\natexlab{a}}, pp. \bibinfo{pages}{406--416}.
\bibitem[{Pelantov\'a and Starosta(2011{\natexlab{b}})}]{PeSta2}
\bibinfo{author}{E.~Pelantov\'a}, \bibinfo{author}{{\v{S}}.~Starosta},
  \bibinfo{title}{Palindromic richness and coxeter groups},
  \bibinfo{journal}{preprint available at http://arxiv.org/abs/1108.3042}
  (\bibinfo{year}{2011}{\natexlab{b}}).
\bibitem[{Starosta(2011)}]{Sta2010}
\bibinfo{author}{{\v{S}}.~Starosta}, \bibinfo{title}{On theta-palindromic
  richness}, \bibinfo{journal}{Theoret. Comput. Sci.} \bibinfo{volume}{412}
  (\bibinfo{year}{2011}) \bibinfo{pages}{1111--1121}.
\bibitem[{Stein et~al.(2011)}]{sage_4_6}
\bibinfo{author}{W.~Stein}, et~al., \bibinfo{title}{{S}age {M}athematics
  {S}oftware (version 4.6)}, \bibinfo{organization}{The Sage Development Team},
  \bibinfo{year}{2011}. \bibinfo{note}{\texttt{http://www.sagemath.org}}.

\end{thebibliography}

\end{document}